\documentclass[10pt]{article}
\usepackage{amsfonts,amssymb,mathrsfs,dsfont}
\usepackage[usenames,dvipsnames]{color}
\usepackage{amsmath}

\usepackage[citecolor=MidnightBlue,linkcolor=MidnightBlue,urlcolor=MidnightBlue,colorlinks=true]{hyperref}
\usepackage[top=1in, bottom=1in, left=1.25in, right=1.25in]{geometry}

\newtheorem{theorem}{Theorem}[section]

\newtheorem{definition}[theorem]{Definition}
\newtheorem{example}[theorem]{Example}
\newtheorem{lemma}[theorem]{Lemma}
\newtheorem{proposition}[theorem]{Proposition}
\newtheorem{remark}[theorem]{Remark}
\newenvironment{proof}[1][Proof]{\noindent \textbf{#1.} }{\  $\Box$}
\numberwithin{equation}{section}
\begin{document}

\title{BSDEs driven by $G$-Brownian motion with  uniformly continuous generators}
\author{
Falei Wang \thanks{Zhongtai Securities Institute for Finance Studies and Institute for Advanced Research, Shandong University,
flwang@sdu.edu.cn. Research partially supported by  the National Natural Science Foundation of China  (No. 11601282) and the Natural Science Foundation of  Shandong Province (No. ZR2016AQ10).}
\and  Guoqiang Zheng\thanks{
School of Mathematics, Southeast  University, gqzheng@seu.edu.cn. Research partially  supported by the Fundamental Research Funds for the Central Universities (No. 3207018202) and
the National Natural Science Foundation of China (No. 11671231 and No. L162400033). } }
\maketitle
\date{}

\begin{abstract}
The present paper is devoted to investigating the existence and uniqueness of solutions to a class of non-Lipschitz  scalar valued  backward stochastic differential equations driven by $G$-Brownian motion ($G$-BSDEs). In fact, when the generators are Lipschitz continuous in $y$ and uniformly continuous in $z$, we construct the unique solution  to such equations by monotone convergence argument. The comparison theorem and related Feynman-Kac formula are stated as well.
\end{abstract}


\section{Introduction}
Given a Wiener space $(\Omega,\mathcal{F},\{\mathcal{F}_t\}_{t\in[0,\infty)},P_0)$, under which the canonical process $W_t$ constitutes a Brownian motion.
A typical nonlinear Lipschitz backward stochastic differential equations (BSDEs), which is formulated in Pardoux and Peng \cite{PP}, takes the form,
$$  Y_t=\xi +\int_t^T f(s,Y_s,Z_s)ds-\int_t^T  Z_s dW_s.$$
The authors found  a unique pair of adapted processes $(Y_s,Z_s)_{0\leq s\leq T}$ that satisfy the above equation for given squarely integrable  terminal value $\xi$ and Lipschitz generator $f$.

From then on, extensive  efforts have been made towards relaxing the Lipschitz  assumptions on the generator $f$. To mention just a few, for the scalar case, i.e., when $Y$ is 1-dimensional, Lepeltier and Martin \cite{LM} confirmed the existence of  solutions to BSDEs with continuous generator that is of linear growth. Kobylanski \cite{K} developed  the existence  for BSDEs with continuous generator that has a quadratic growth in $z$ when the terminal value $\xi$ is bounded. Also for the quadratic cases, Briand and Hu \cite{BH2006,BH2008} successively obtained the existence of solution for unbounded $\xi$.
More imaginative works on generalizing  the classical BSDEs theory from    different points of view   are emerging, and it would be too ambitious for us to give  an overview of all variants.
In this paper, we  focus on the study  of BSDEs under $G$-expectation framework.

The $G$-expectation theory was put forth by Peng \cite{P07,P08,P10}, which provides a unified  tool for stochastic analysis problems that involve non-dominated family of probability measures. In particular, the $G$-Brownian motion process is constructed with uncertain  quadratic variation process, a  feature that is helpful in capturing  the volatility fluctuations of financial market.
However there are many challenges in the research of $G$-expectation due to the uncertainty, for instance, the general dominated convergence theorem  is not available, see Example 11 in \cite{HJ1}.
Furthermore,  there exist non-increasing and continuous $G$-martingales called non-symmetrical $G$-martingales, which
 makes the $G$-martingale representation theorem  more difficult, see \cite{P10, STZ, Song11}.

 Recently, Hu, Ji, Peng and Song  \cite{HJPS2014} considered  the well-posedness problem of BSDEs driven by $G$-Brownian motion $B$:
 \begin{equation}\label{gbsd}
 Y_t=\xi +\int_t^T f(s,Y_s,Z_s)ds+\int_{t}^{T}g^{ij}(s,Y_s,Z_s)d\langle B^i, B^j\rangle_{s}-\int_t^T  Z_s dB_s- (K_T-K_t),
 \end{equation}
 where $\langle B^i, B^j\rangle$ denotes the quadratic (co)variation process and the generators $f,g$ are  Lipschitz continuous in $(y,z)$.
 The solution of  $G$-BSDE \eqref{gbsd} consists of a triple of processes $(Y,Z,K)$, where $K$ is a   non-symmetrical $G$-martingale.
 Note that  the classical Banach contraction mapping principle cannot be applied directly to this  equation due to the existence of $K$. The authors use PDE techniques and an   approximation of Galerkin type
 to obtain the existence and uniqueness result of $G$-BSDE \eqref{gbsd}. In an  accompany paper \cite{HJPS2014C}, they established the comparison theorem, Girsanov theorem and the relevant  nonlinear Feynman-Kac formula.

Note that there are at least two characteristics that make  the study  on $G$-BSDEs meaningful, for one thing, we could establish a connection with fully nonlinear parabolic partial differential equations (PDEs) using $G$-BSDEs, for the other, since there exists a family of non-dominated, mutually singular martingale measures underlying the $G$-Brwonian motion,  one can solve simultaneously a family of  classical BSDEs  driven by mutually singular continuous martingales through dealing with only one aggregated $G$-BSDE.  Moreover,  Song \cite{Song14} obtained gradient estimates for certain   nonlinear  partial differential equations (PDEs) by combining $G$-expectation theory with  coupling methods.
A close  approach to $G$-BSDEs is the so-called second order BSDEs framework proposed   independently by Soner, Touzi and Zhang \cite{STZ1}.

 Still there are further research papers on getting rid of the Lipschitz  assumptions, and extensions of $G$-BSDEs from different aspects. For instance, Hu, Lin and Soumana Hima \cite{HLH}
studied the $G$-BSDEs under quadratic assumptions on coefficients and Li, Peng and Soumana Hima \cite{LPH} considered $G$-BSDE with reflection,  for which situation the solution is forced
to lie above a prescribed continuous process. This paper is devoted to the research of the existence of solution to equation \eqref{gbsd} when  $f,g$ are Lipschitz continuous in $y$ and uniformly continuous in $z$, yet with a linear growth in both arguments.

In classical situation, the  term $B$ in \eqref{gbsd} boils down to the standard Wiener process, $d\langle B^i,B^j\rangle_t=(dt)\mathbf{1}_{i=j}$, and the term $K$ vanishes.
It was  due to Lepeltier and Martin \cite{LM} that confirmed those  equations  allow for solutions by monotone convergence argument, indeed their results hold for BSDEs that have continuous coefficients with linear growth,  and Jia \cite{J,J2010} supplemented with  proofs on uniqueness of solution.
However Lepeltier and Martin's arguments  cannot be applied directly to investigate the existence of solution to equation \eqref{gbsd},
 because the monotone convergence theorem  of $G$-expectation is hardly at hand and the convergence of approximating sequences of $G$-BSDEs is not obvious.

Our observation is,  this obstacle can be overcome with the help  of a uniform estimate for approximating sequences of $G$-BSDEs, see Lemma \ref{lemma2}.
Indeed, the  uniformly continuous generators can be approximated uniformly by a sequence of Lipschitz generators (see \cite{J2010}),
 from which we could prove the convergence of  approximating sequences of $G$-BSDEs based on the linearization method of \cite{HJPS2014C} and \cite{HW}.
Then we obtain  the existence and uniqueness of the solution to $G$-BSDE \eqref{gbsd} by $G$-stochastic analysis technique.
Since our work relies heavily on the comparison theorem of  Lipschitz $G$-BSDEs, we can only deal with the one dimensional case.
And the comparison theorem still holds for this type of $G$-BSDEs.
 Finally the connection of this equations with the second order fully nonlinear PDEs  are discussed,
 thanks to the stability of viscosity solution, we show the solution to Markovian $G$-BSDE \eqref{gbsd} defines the unique solution to the related  PDE, in the spirit of Feynman-Kac formula.

 The paper is organized as follows.  In the Section 2, we provide with preliminary notions on $G$-expectation and Lipschitz $G$-BSDEs.
 In Section 3, we state and prove our main theorem and the comparison theorem of our version. As an application, a slightly more general form of nonlinear Feynman-Kac formula is obtained in Section 4.

\section{Preliminaries}
To begin with, we shall recall some ingredients of $G$-expectation theory mainly from the seminal work of Peng \cite{P10}, and then of $G$-BSDEs results from \cite{HJPS2014,HJPS2014C}.

\subsection{$G$-expectation }
Consider the canonical path space $\Omega=C_{0}([0,\infty),\mathbb{R}^d)$, all continuous paths $\omega$ vanishing at zero, i.e., $\omega_0=0$. $(\Omega,\rho)$ is readily seen to be a complete separable metric space, where $\rho$ is given by,
\[
\rho(\omega^{1},\omega^{2}):=\sum_{i=1}^{\infty}2^{-i}[(\max_{t\in
\lbrack 0,i]}|\omega_{t}^{1}-\omega_{t}^{2}|)\wedge1].
\]
In the sequel, we will make use of these notations,
\begin{itemize}
\item  $B$ denotes the $d$-dimensional canonical process, i.e. $B_t(\omega)=\omega_t$, for any $\omega\in\Omega$.
\item $\mathcal{B}(\Omega)$ denotes the Borel $\sigma$-algebra of $\Omega$, similarly we have $\mathcal{B}(\Omega_{t})$ with  $\Omega_{t}:=\{ \omega_{\cdot \wedge t}:\omega \in \Omega \}$.
\item $L_{ip}(\Omega):=\{ \varphi(B_{t_{1}},\ldots,B_{t_{k}}):k\in \mathbb{N}
,t_{1},\ldots,t_{k}\in \lbrack0,\infty),\varphi \in C_{b.Lip}(\mathbb{R}
^{k\times d })\}$, where $C_{b.Lip}(\mathbb{R}^{k\times d})$ collects all
bounded Lipschitz functions on $\mathbb{R}^{k\times d}$, $L_{ip}(\Omega_{t})$ denotes all $\mathcal{B}(\Omega_{t})$-measurable elements in $L_{ip}%
(\Omega)$.
\item $\mathbb{S}_{d}$ denote all symmetric matrices of size $d$.
\end{itemize}

For any given monotonic sublinear   continuous function $G:\mathbb{S}_d\rightarrow \mathbb{R}$, Peng \cite{P10} associated it with a nonlinear $G$-expectation $\mathbb{\hat{E}}[\cdot]$ using a nonlinear parabolic PDE,
which in turn makes the canonical processes $B$ a $d$-dimensional $G$-Brownian motion, ending up with the so-called $G$-expectation space $(\Omega,L_{ip}(\Omega),\mathbb{\hat{E}%
}[\cdot],(\mathbb{\hat{E}}_{t}[\cdot])_{t\geq0})$. The readers are referred to \cite{P07,P08,P10} for detailed construction and so forth.

  For each
$p\geq1$, the completion of $L_{ip}(\Omega)$ under the norm
$||X||_{L_{G}^{p}}:=(\mathbb{\hat{E}}[|X|^{p}])^{1/p}$ is denoted by
$L_{G}^{p}(\Omega)$. Similarly, we can define $L_{G}^{p}(\Omega_t)$ for each $t>0$.
 The $G$-expectation $\mathbb{\hat{E}}[\cdot]$ and conditional $G$-expectation can be
extended continuously to the completion  $L_{G}^{p}(\Omega)$.
And the $G$-expectation can be regarded as a upper expectation.
\begin{theorem}[\cite{DHP2011,HP09}]
\label{the2.7}  There exists a weakly compact set
$\mathcal{P}$ of probability
measures on $(\Omega,\mathcal{B}(\Omega))$, such that
\[
\mathbb{\hat{E}}[\xi]=\sup_{P\in\mathcal{P}}E_{P}[\xi],\  \text{for
 all}\ \xi\in  {L}_{G}^{1}{(\Omega)}.
\]
\end{theorem}

For this $\mathcal{P}$, we define a capacity
\[
c(A):=\sup_{P\in\mathcal{P}}P(A),\ A\in\mathcal{B}(\Omega).
\]
A set $A\in\mathcal{B}(\Omega)$ is polar if $c(A)=0$.  A
property holds $``quasi$-$surely"$ (q.s.) if it holds except for a
polar set. In what follows,  we do not distinguish two random
variables between  $X$ and $Y$, if $X=Y$ q.s..

Now we state the nonlinear monotone convergence theorem, which is different from the linear case.
\begin{proposition}[\cite{DHP2011}]\label{downward convergence proposition} Suppose $X_n$, $n\geq 1$ and $X$ are $\mathcal{B}(\Omega)$-measurable.
If $\{X_n\}_{n=1}^\infty$ in ${L}_{G}^{1}(\Omega)$ satisfies that $X_n\downarrow X$, q.s.,  then
$\mathbb{\hat{E}}[X_n]\downarrow\mathbb{\hat{E}}[X].	$
\end{proposition}

Peng also introduced the stochastic integral with respect to $G$-Brownian motion, which led to a symmetric $G$-martingale.
Given a fixed contant $T>0$, the following spaces of stochastic processes will be useful,
\begin{itemize}
 \item $M_{G}^{0}(0,T):=\{\eta:\ \eta_{t}(\omega)=\sum\limits_{j=0}^{N-1}\xi_{j}(\omega)I_{[t_{j},t_{j+1})}(t),\  \xi_{i}\in L_{ip}(\Omega_{t_{i}})$ for some partition $t_{0}\leq t_1\leq \ldots\leq t_{N}$ of $[0,T]\}.$
 \item $M_{G}^{2}(0,T)$ is the completion of  $M_{G}^{0}(0,T)$ under norm $\Vert \eta \Vert_{M_{G}^{2}}=\{ \mathbb{\hat{E}}[\int_{0}^{T}|\eta_{s}|^{2}ds]\}^{1/2}.$
 \item $S_{G}^{0}(0,T)=\{h(t,B_{t_{1}\wedge t},\cdot \cdot \cdot,B_{t_{n}\wedge t}):t_{1},\ldots,t_{n}\in \lbrack0,T],h\in C_{b,Lip}(\mathbb{R}^{n+1})\}$.
 \item $S_{G}^{2}(0,T)$ is the completion of $S_{G}^{0}(0,T)$ under the norm $\Vert \eta\Vert_{S_{G}^{2}}=\{ \mathbb{\hat{E}}[\sup_{t\in\lbrack0,T]}|\eta_{t}|^{2}]\}^{\frac{1}{2}}$.
\end{itemize}

For each $1\leq i, j \leq d$,   we denote by $\langle B^i, B^j\rangle$   the mutual  variation process.
Then for two processes $ \eta\in M_{G}^{2}(0,T)$ and $ \xi\in M_{G}^{1}(0,T)$,
the $G$-It\^{o} integrals  $\int^{\cdot}_0\eta_sdB^i_s$ and $\int^{\cdot}_0\xi_sd\langle
B^i,B^j\rangle_s$  are well defined, see Peng \cite{P10}.
Moreover,   the corresponding $G$-It\^{o} formula were established.

\subsection{ $G$-BSDEs with Lipschitz assumptions}
From now on, we always assume that the function $G$  is  non-degenerate throughout our paper, i.e. there are two constants $0<\underline{\sigma}^{2}\leq\bar{\sigma}^{2}<\infty$ such that
$$ \frac{1}{2}\underline{\sigma}^2\text{tr}[A-A']\leq G(A)-G(A')\leq \frac{1}{2}\bar{\sigma}^2\text{tr}[A-A'], \text{ for } A\geq A'.$$
Then there exists a bounded and closed subset $\Gamma
\subset\mathbb{S}^+(d)$ such that
\begin{equation*}
G(A)=\frac{1}{2}\underset{Q\in\Gamma}{\sup}\mathrm{tr}[AQ],
\label{G-representation}
\end{equation*}
where $\mathbb{S}^+(d)$ denotes the space of all $d\times d$ symmetric positive-definite matrices.

Consider the following $G$-BSDEs (recall  that we use Einstein summation convention):
\begin{equation}\label{GBSDE}
 Y_t=\xi +\int_t^T f(s,Y_s,Z_s)ds+\int_{t}^{T}g^{ij}(s,Y_s,Z_s)d\langle B^i, B^j\rangle_{s}-\int_t^T  Z_s dB_s- (K_T-K_t),
 \end{equation}
in which the generators $$f(t,\omega,y,z), g^{ij}(t,\omega,y,z): [0,T]\times\Omega_T\times\mathbb{R}\times\mathbb{R}^d\rightarrow\mathbb{R}$$ and the terminal value $\xi$ fulfill these assumptions,
\begin{flalign*}
&\textbf{(A1)} \text{ For some }  \beta >2, \xi\in L_G^\beta(\Omega_T); \text{ for any }(y,z), f(\cdot, \cdot,y,z)\in M_G^\beta(0,T),g^{ij}(\cdot, \cdot,y,z)\in M_G^\beta(0,T).   &     \\
&\textbf{(A2)} \text{ There is a Lipschitz constant } L_0>0, \text{so that}    &
\end{flalign*}
$$ \vert f(t,y,z)-f(t,y',z')\vert +\vert g^{ij}(t,y,z)-g^{ij}(t,y',z')\vert  \leq L_0\vert y-y'\vert +L_0 \vert z-z'\vert.$$

For simplicity, we denote by $\mathfrak{S}_G^{2}(0,T)$ the collection of process $(Y,Z,K)$ such that $Y\in S_G^{2}(0,T)$, $Z\in M_G^{2}(0,T;\mathbb{R}^d)$, $K$ is a non-increasing $G$-martingale with $K_0=0$ and $K_T\in L_G^{2}(\Omega_T)$. Hu et al. \cite{HJPS2014,HJPS2014C} firstly obtained the existence and uniqueness result on Lipschitz $G$-BSDEs \eqref{GBSDE}, and the  comparison principle.
\begin{theorem}[\cite{HJPS2014}]\label{G1}
 Assume the conditions $\textbf{(A1)}$ and $\textbf{(A2)}$ hold.
   Then the equation $(\ref{GBSDE})$ admits a unique solution  $(Y,Z,K)\in \mathfrak{S}_G^{2}(0,T)$.
 \end{theorem}
\begin{theorem}[\cite{HJPS2014C}]\label{my2}
Assume $(\xi^{\nu},f^{\nu},g^{\nu}_{ij})$ satisfy assumption $\textbf{(A1)}$  for $\nu=1,2$.
Moreover,  one of them  satisfies assumption $\textbf{(A2)}$.
Suppose $(Y^{\nu},Z^{\nu},K^{\nu})$ is a $\mathfrak{S}_G^{2}(0,T)$-solution to the $G$-BSDE \eqref{GBSDE} with data $(\xi^{\nu},f^{\nu},g^{\nu}_{ij})$.
 If $\xi^2\leq \xi^1,$ $f^2\leq f^1$ and the matrix $(g^2_{ij})_{i,j=1}^d\leq (g^1_{ij})_{i,j=1}^d$, then we have $Y_t^2\leq Y_t^1$ for all $t\in [0,T].$
\end{theorem}

The     linear   $G$-BSDEs will be repeatedly used in our paper, so we sketch the idea on how to construct the solution. Consider linear $G$-BSDE of the form,
\begin{equation}\label{1LBSDE1}
 Y_t=\xi +\int_t^T [a_{s}Y_{s}+b_{s}Z_{s}+m_{s}]ds+\int_{t}^{T}[c_{s}^{ij}Y_{s}+d_{s}^{ij}Z_{s}+n^{ij}_{s}]d\langle B^i,B^j\rangle_{s}-\int_t^T  Z_s dB_s- (K_T-K_t),
 \end{equation}
  where $(a_{s})_{s\in\lbrack0,T]}$,
$(c^{ij}_{s})_{ s\in\lbrack0,T]}\in M_{G}^{2}(0,T)$, $(b_{s})_{s\in\lbrack0,T]}$,$(d^{ij}_{s})_{s\in\lbrack0,T]}\in M_{G}^{2}(0,T;\mathbb{R}^d)$ are bounded
processes and $\xi\in L_{G}^{2}(\Omega_{T})$,
$(m_{s})_{s\in\lbrack0,T]}$, $(n^{ij}_{s})_{s\in\lbrack0,T]}\in M_{G}^{2}(0,T)$.

To find the closed-form solution  to equation (\ref{1LBSDE1}), a standard method is to introduce a dual process. However for  the $G$-expectation case, unless the $G$-Brownian motion degenerates to the standard Wiener process,  the measures $ds$ and $d\langle B\rangle_{s}$ are mutually singular, therefore to cancel terms involving  $ds$ and $d\langle B^i,B^j\rangle_{s}$  is even harder.
To adapt the classical dual method, Hu et al. \cite{HJPS2014C} came up with a strategy of enlarging the original $G$-expectation space to
$\tilde{G}$-expectation space $(\tilde{\Omega},L_{\tilde{G}}^{1}
(\tilde{\Omega}),\mathbb{\hat{E}}^{\tilde{G}})$ with $\tilde{\Omega}=C_{0}([0,\infty),\mathbb{R}^{2d})$ and
\begin{align}\label{yw7}
\tilde{G}(A)=\frac{1}{2}\sup_{Q\in\Gamma}\mathrm{tr}\left[  A\left[
\begin{array}
[c]{cc}
Q & I_d\\
I_d & Q^{-1}
\end{array}
\right]  \right]  ,\ A\in\mathbb{S}_{2d}.
\end{align}
Let $(B_{t},\tilde{B}_{t})_{t\geq 0}$ be the canonical process in the extended space. Then
\begin{lemma}[\cite{HJPS2014C}]
\label{the5.2} In the extended $\tilde{G}$-expectation space, the solution of
the linear $G$-BSDE (\ref{1LBSDE1}) can be represented as
\begin{equation*}
Y_{t}=\mathbb{\hat{E}}_{t}^{\tilde{G}}[\tilde{\Gamma}_{T}^t\xi+\int_{t}^{T}
m_{s}\tilde{\Gamma}_s^tds+\int_{t}^{T}n_{s}^{ij}\tilde{\Gamma}^t_sd\langle B^i,B^j\rangle_{s}],
\end{equation*}
where $\{\tilde{\Gamma}^t_s\}_{s\in\lbrack t,T]}$ is the solution of the  following $\tilde{G}$-SDE:\begin{equation}
\tilde{\Gamma}^t_s=1+\int_{t}^{s}a_{r}\tilde{\Gamma}^t_rdr+\int_{t}^{s}c^{ij}_{r}\tilde{\Gamma}^t_rd\langle B^i, B^j\rangle
_{r}+\int_{t}^{s}d_{r}^{ij}\tilde{\Gamma}^t_rdB_{r}+\int_{t}^{s}b_{r}\tilde{\Gamma}^t_{r}d\tilde{B}_{r}.
\label{LSDE2}
\end{equation}
Moreover,
\begin{equation}\label{yw2}
\mathbb{\hat{E}}_{t}^{\tilde{G}}[\tilde{\Gamma}^t_TK_{T}-\int_{t}^{T}a_{s}K_{s}\tilde{\Gamma}^t_s
ds-\int_{t}^{T}c^{ij}_{s}K_{s}\tilde{\Gamma}^t_sd\langle
B^i, B^j\rangle_{s}]=K_{t}.
\end{equation}
\end{lemma}

\section{$G$-BSDEs with uniformly continuous generators}
In this section, we  shall investigate  the well-posedness problem of the subsequent $G$-BSDEs
\begin{equation}\label{GBSDE1}
 Y_t=\xi +\int_t^T f(s,Y_s,Z_s)ds+\int_t^T g^{ij}(s,Y_s,Z_s)d\langle B^i,B^j\rangle_s-\int_t^T  Z_s dB_s- (K_T-K_t),
 \end{equation}
 where the generators
\begin{equation*}
f(t,\omega,y,z),\  g^{ij}(t,\omega,y,z):[0,T]\times \Omega\times \mathbb{R}\times\mathbb{R}^d\rightarrow
\mathbb{R},
\end{equation*}
satisfy the following assumptions:
\begin{description}
\item[(H1)]  There exists a constant $ \beta >2 $ such that $f(\cdot, \cdot,y,z), g(\cdot,\cdot,y,z) \in M_G^\beta(0,T)$  for any $ y,z$.
\item[(H2)]   $f$ {and} $ g $ are Lipschitz continuous in $ y$, are of linear growth and  uniformly
 continuous in $ z$, i.e. there is a constant $ L$ and a continuous function $\phi$,  both independent of $(t,\omega)$,  such that
   $$ \vert f(t,\omega,y,z)-f(t,\omega,y',z')\vert+\vert g^{ij}(t,\omega,y,z)-g^{ij}(t,\omega,y',z')\vert\leq L\vert y-y'\vert+\phi(\vert z-z'\vert), $$
   where $\phi:\mathbb{R^+}\rightarrow \mathbb{R^+}$ is nondecreasing and sub-additive, with $\phi(0)=0$ as well as
   $\phi(z)\leq L(1+\vert z\vert).$
\item[(H3)]  $g^{ij}\equiv 0$ whenever $i\neq j$.
\end{description}

\begin{remark}
{\upshape
Note that assumption (H3) is necessary to construct a sequence of $G$-BSDEs monotonically converges to $Y$, see (i) of Lemma \ref{lemma2}.
}
\end{remark}

According to  Lemma 1 in Lepeltier and Martin \cite{LM} or Lemma 2 in Jia \cite{J},
there exists a sequence of  Lipschitz functions that nicely approximates $f$ and $g_{ij}$ respectively. Indeed,
for  any $(t,y,z),n\in\mathbb{N}$ and for every $\omega$, denote
\begin{align*}
\underline{\varphi}_n(t,y,z):=\inf\limits_{q\in\mathbb{Q}}\{\varphi(t,y,q)+n\vert z-q\vert\}-\varphi_0(t),\ \
\bar{\varphi}_n(t,y,z):=\sup\limits_{q\in\mathbb{Q}}\{\varphi(t,y,q)-n\vert z-q\vert\}-\varphi_0(t),
\end{align*}
where $\varphi=f, g^{ij}$ and $\varphi_0(t)=\varphi(t,0,0)$. Their main technical lemma can be summarized  as,
\begin{lemma}
\label{lemma1}
 Assume \textbf{(H1)-(H2)} hold. Then  for each $n>L$, the following properties hold
 \begin{description}
\item[(i)]  both $ \underline{\varphi}_n$ and $\bar{\varphi}_n$ are of linear growth, moreover, for all $(t,y,z)$,
    $$ -L(1+\vert y\vert+\vert z\vert)\leq \underline{\varphi}_n(t,y,z)\leq \varphi(t,y,z)-\varphi_0(t)\leq \bar{\varphi}_n(t,y,z) \leq  L(1+\vert y\vert+\vert z\vert);$$
\item[(ii)]  for all $(t,y,z)$,   $\underline{\varphi}_.(t,y,z)$ is  non-decreasing and $\bar{\varphi}_.(t,y,z)$ is non-increasing;
\item[(iii)]$\underline{\varphi}_n(t,y,\cdot)$ and $\bar{\varphi}_n(t,y,\cdot)$ are Lipschitz functions with constant $n$,
$\underline{\varphi}_n(t,\cdot,z)$ and $\bar{\varphi}_n(t,\cdot,z)$ are Lipschitz functions with Lipschitz constant $L$;
\item[(iv)] if $(y_n,z_n)\rightarrow (y,z),$  then $\underline{\varphi}_n(t,y_n,z_n)\rightarrow \varphi(t,y,z)-\varphi_0(t)$ and $\bar{\varphi}_n(t,y_n,z_n)\rightarrow \varphi(t,y,z)-\varphi_0(t)$;
\item[(v)] for all $(t,\omega,y,z)$,
$$0\leq \varphi(t,y,z)-\varphi_0(t)-\underline{\varphi}_n(t,y,z)\leq \phi(\frac{2L}{n-L}) , 0\leq \bar{\varphi}_n(t,y,z)+\varphi_0(t)-\varphi(t,y,z)\leq \phi(\frac{2L}{n-L}). $$
\end{description}
\end{lemma}

Based on the above  approximation results, we  construct two sequences of $G$-BSDEs  corresponding respectively to $(\underline{f}_n,\underline{g}_{n}^{ij})$ and $(\bar{f}_n,\bar{g}_{n}^{ij})$, i.e.,
\begin{align}\label{my1}
\begin{split}
\underline{Y}^n_t=&\xi +\int_t^T [\underline{f}_n(s,\underline{Y}^n_s, \underline{Z}^n_s)+f_0(s)]ds+\int_t^T[ \underline{g}^{ij}_n(s,\underline{Y}^n_s, \underline{Z}^n_s)+g^{ij}_0(s)]d\langle B^i,B^j\rangle_s -\int_t^T  \underline{Z}^n_s dB_s\\
& \ \ \ \ \ - (\underline{K}^n_T-\underline{K}^n_t),\\
 \bar{Y}^n_t=&\xi +\int_t^T [\bar{f}_n(s,\bar{Y}^n_s, \bar{Z}^n_s)+f_0(s)]ds+\int_t^T [\bar{g}_n^{ij}(s,\bar{Y}^n_s, \bar{Z}^n_s)+g^{ij}_0(s)]d\langle B^i,B^j\rangle_s-\int_t^T  \bar{Z}^n_s dB_s \\
& \ \ \ \ \ -(\bar{K}^n_T-\bar{K}^n_t).
\end{split}
\end{align}

We need an additional assumption to ensure the existence of $\underline{Y}^n$ and $ \bar{Y}^n$:
\begin{description}
\item[(H4)] For each $n$ and for any $(y,z)$, $\underline{\varphi}_n(t,y,z)$ and $\bar{\varphi}_n(t,y,z)$ all belong to  $M^{\beta}_G(0,T)$, with $\varphi=f,g^{ij}$.
\end{description}

\begin{remark}
{\upshape As can be easily seen, assumption \textbf{(H4)} is imposed mainly to keep all processes under investigation lying in   space $M^{\beta}_G(0,T)$.
This condition can be verified for lots of situations.  For instance, assume  \textbf{ (H1)} hold.
Suppose for $\varphi=f,g$ that $\varphi(\cdot, \cdot, y, z)$ is uniformly continuous in $(t, \omega)$ with the modulus of continuity
  independent of $(y, z)$,
\[
\vert \varphi(t,\omega,y,z)-\varphi(t',\omega',y,z)\vert\leq \phi(|t-t'|+\sup\limits_{s\in[0,t]}|\omega(s)-\omega'(s)|).
\]
Then it is straightforward  to observe that $\underline{\varphi}_n(\cdot,\cdot,y,z)$ and $\bar{\varphi}_n(\cdot,\cdot, y,z)$ are uniformly continuous in $(t,\omega)$.
Recalling the property (i) from  Lemma \ref{lemma1}, we have $\underline{\varphi}_n(\cdot,\cdot,y,z)$ and $\bar{\varphi}_n(\cdot,\cdot, y,z)$ are bounded  and then \[
\lim\limits_{N\rightarrow\infty}\mathbb{\hat{E}}[\int^T_0|\underline{\varphi}_n(t,y,z)|^{\beta}\mathbf{1}_{\{|\underline{\varphi}_n(t,y,z)|\geq N\}}dt]=0, \ \ \lim\limits_{N\rightarrow\infty}\mathbb{\hat{E}}[\int^T_0|\bar{\varphi}_n(t,y,z)|^{\beta}\mathbf{1}_{\{|\bar{\varphi}_n(t,y,z)|\geq N\}}dt]=0.
\]
Thus by Theorem 4.16 in \cite{HWZ2016}, we know $\varphi$ satisfies assumption \textbf{(H4)}.
}
\end{remark}

The following lemma is important in our future discussion.

\begin{lemma}\label{lemma2}
Let $\xi$ be in $L^{\beta}_G(\Omega_T)$ and  the assumptions \textbf{ (H1)-(H4)} hold. Then  the $G$-BSDE \eqref{my1} has a unique $\mathfrak{S}_G^{2}(0,T)$-solution. Moreover, we have
\begin{description}
\item[(i)]  for any $n,m\in\mathbb{N}$,   the comparisons $\underline{Y}^n\leq\underline{Y}^{n+1}\leq\bar{Y}^{m+1}\leq\bar{Y}^m$  hold;
\item[(ii)]  both $\underline{Y}^n$ and $\bar{Y}^n$ are uniformly bounded in $S_G^2(0,T)$;
\item[(iii)]  for each $n>L$,
the differences between $\underline{Y}^n$ and $\bar{Y}^n  $ can be uniformly controlled, that is,
$$\vert \underline{Y}^n_t-\bar{Y}^n_t\vert\leq C_G\phi(\frac{2L}{n-L}), \ \forall t\in[0,T],$$
where $C_G$ is a constant depending on $G,L$ and $T$.
\end{description}
\end{lemma}
\begin{proof}
The proof is built on the conclusions of lemma \ref{lemma1}. By assumption \textbf{(H3)}, we have $\underline{g}^{ij}_n=\bar{g}^{ij}_n=0$ whenever $i\neq j$.
Thus $(\underline{g}^{ij}_n-\underline{g}^{ij}_{n+1})_{i,j=1}^d$ is a nonnegative definite matrix.
Then from Theorem \ref{G1} and the comparison theorem \ref{my2}, it is trivial to verify  (i)   in  view of assertions (i)-(ii) from lemma \ref{lemma1}.

In order to prove (ii), setting $w(y,z)=L(1+|y|+|z|)$,  consider the following $G$-BSDEs
\begin{align*}
&U_t=\xi +\int_t^T [w( U_s, V_s)+f_0(s)] ds+\int_t^T [w( U_s, V_s)+g_0^{ii}(s)] d\langle B^i,B^i\rangle_s-\int_t^T  V_s dB_s- (R_T-R_t),\\
&U'_t=\xi +\int_t^T [-w( U'_s, V'_s)+f_0(s)] ds+\int_t^T [-w( U'_s, V'_s)+g_0^{ii}(s)] d\langle B^i,B^i\rangle_s-\int_t^T  V'_s dB_s- (R'_T-R'_t).
\end{align*}
It follows from Theorem \ref{G1} that the above $G$-BSDE admits a unique $\mathfrak{S}_G^{2}(0,T)$-solution $(U,V,R)$ and $(U',V',R')$, respectively.
Then by  (i) of lemma \ref{lemma1} and the comparison theorem \ref{my2},  it holds that for any $n\in\mathbb{N}$
\[
U'_t\leq \underline{Y}^n_t\leq\bar{Y}^n_t \leq U_t, \ \forall t\in[0,T], \]
which implies the desired result.

Finally, we proceed to verify the third assertion (iii).  Without loss of generality, assume  that  $d=1$.
Set $(\hat{Y},\hat{Z})=(\bar{Y}^n-\underline{Y}^n,\bar{Z}^n-\underline{Z}^n)$.
Then for each $t\in[0,T]$, we have
 \begin{equation}\label{eq5}
\hat{Y}_t+\underline{K}^n_t=\underline{K}^n_T+\int_t^T \hat{f}_sds+\int^T_t\hat{g}_sd\langle B\rangle_s-\int_t^T  \hat{Z}_s  dB_s- (\bar{K}^n_T-\bar{K}^n_t),
 \end{equation}
where $ \hat{\varphi}_s=\bar{\varphi}_n(s,\bar{Y}^n_s, \bar{Z}^n_s)-\underline{\varphi}_n(s,\underline{Y}^n_s, \underline{Z}^n_s)$ for $\varphi=f,g.$

By Lemma 3.5 in \cite{HW}, for each $\varepsilon>0$, there exist four bounded processes $a^{\varepsilon}, b^{\varepsilon}$, $c^{\varepsilon}$, $d^{\varepsilon}\in M^2_G(0,T)$
such that for all $s\in\lbrack 0,T\rbrack,$
\[
\hat{f}_{s}=a_{s}^{\varepsilon}\hat{Y}_{s}+b_{s}^{\varepsilon}\hat{Z}
_{s}+m_s-m_{s}^{\varepsilon},\ \hat{g}_{s}=c_{s}^{\varepsilon}\hat{Y}
_{s}+d_{s}^{\varepsilon}\hat{Z}_{s}+n_s-n_{s}^{\varepsilon},
\]
and $|a_{s}^{\varepsilon}|\leq L$, $|c_{s}^{\varepsilon}|\leq L$,
$|b_{s}^{\varepsilon}|\leq n$, $|d_{s}^{\varepsilon}|\leq n$, $|m_{s}^{\varepsilon}|\leq2(L+n)\varepsilon$, $|n_{s}^{\varepsilon}%
|\leq2(L+n)\varepsilon$, $m_{s}=\bar{f}_n(s,\underline{Y}^n_s, \underline{Z}^n_s)-\underline{f}_n(s,\underline{Y}^n_s, \underline{Z}^n_s)$,
$n_{s}=\bar{g}_n(s,\underline{Y}^n_s, \underline{Z}^n_s)-\underline{g}_n(s,\underline{Y}^n_s, \underline{Z}^n_s)$.

In order to estimate the solution to  the above linearized equation (\ref{eq5}), as in \cite{HJPS2014C}, we
 shift from the underlying $G$-expectation space to an auxiliary extended
$\tilde{G}$-expectation space $(\tilde{\Omega},L_{\tilde{G}}^{1}
(\tilde{\Omega}),\mathbb{\hat{E}}^{\tilde{G}})$ with $\tilde{\Omega}=C_{0}([0,\infty),\mathbb{R}^{2})$, where $\tilde{G}$ is given by equation \eqref{yw7}, within which,
  $(B_{t},\tilde{B}_{t})_{t\geq 0}$  denotes the corresponding canonical process.

Applying Lemma \ref{the5.2}  yields that
\begin{align*}
 \hat{Y}_{t}+\underline{K}^n_t
&  =\mathbb{\hat{E}}_{t}^{\tilde{G}}[\tilde{\Gamma}_{T}
^{t,\varepsilon}\underline{K}_{T}^{n}+\int_{t}^{T}(m_s+2G(n_s)-m_{s}^{\varepsilon
}-a_{s}^{\varepsilon}\underline{K}^n_s)\tilde{\Gamma}_{s}^{t,\varepsilon}ds\\
&  -\int_{t}^{T}(n_{s}^{\varepsilon}+c_{s}^{\varepsilon}\underline{K}^n_s
)\tilde{\Gamma}_{s}^{t,\varepsilon}d\langle B\rangle_{s}+\int_{t}^{T}n_s\tilde{\Gamma}_{s}^{t,\varepsilon}d\langle B\rangle_s-\int_{t}^{T}2G(n_s)\tilde{\Gamma}_{s}^{t,\varepsilon}ds],
\end{align*}
where $\{\tilde{\Gamma}^{t,\varepsilon}_{s}\}_{s\in\lbrack t,T\rbrack}$ is given by, c.f. equation \eqref{LSDE2}. From $G$-It\^{o}'s formula, we conclude  that
\begin{equation*}
\tilde{\Gamma}^{t,\varepsilon}_{s}=\exp(\int_{t}^{s}(a^{\varepsilon}_{r}-b^{\varepsilon}_{r}d^{\varepsilon}_{r})dr+\int_{t}^{s}c^{\varepsilon}_{r}d\langle
B\rangle_{r})\mathcal{E}_{s}^{B}\mathcal{E}_{s}^{\tilde{B}}.
\end{equation*}
Here $\mathcal{E}_{s}^{B}=\exp(\int^s_td^{\varepsilon}_{r}dB_r-\frac{1}{2}\int^s_t|d^{\varepsilon}_{r}|^2d\langle B\rangle_r)$ and
$\mathcal{E}_{s}^{\tilde{B}}=\exp(\int^s_tb^{\varepsilon}_{r}d\tilde{B}_r-\frac{1}{2}\int^s_t|b^{\varepsilon}_{r}|^2d\langle \tilde{B}\rangle_r)$.
Therefore using equations \eqref{eq5} and $G$-It\^{o}'s formula we get that
\begin{align}
 \hat{Y}_{t}+\underline{K}^n_t
  \leq \mathbb{\hat{E}}_{t}^{\tilde{G}}
[\int_{t}^{T}(m_s+2G(n_s))\tilde{\Gamma}_{s}^{t,\varepsilon}ds-\int_{t}^{T}m_{s}^{\varepsilon}\tilde{\Gamma}_{s}^{t,\varepsilon}ds
-\int_{t}^{T}n_{s}^{\varepsilon
}\tilde{\Gamma}_{s}^{t,\varepsilon}d\langle B\rangle_{s}]+\underline{K}^n_t,\ \ q.s..\label{myq4}
\end{align}
By (v) of Lemma \ref{lemma1}, we get
\[
0\leq m_s+2G(n_s)\leq 2(1+\bar{\sigma}^2)\phi(\frac{2L}{n-L}).
\]
Note that for each $s\geq t$, $
\tilde{\Gamma}^{t,\varepsilon}_s\leq \exp(L(1+\bar{\sigma}^2)(s-t)){\Gamma}^{t,\varepsilon}_s,
$
where $
{\Gamma}^{t,\varepsilon}_{s}=1+\int_{t}^{s}d^{\varepsilon}_{r}{\Gamma}^{t,\varepsilon}_{r}dB_{r}+\int_{t}%
^{s}b^{\varepsilon}_{r}{\Gamma}^{t,\varepsilon}_{r}d\tilde{B}_{r}.
$
Then by equation \eqref{myq4},
we derive that
\begin{align*}
\hat{Y}_{t}\leq &
[2(1+\bar{\sigma}^2)\phi(\frac{2L}{n-L})+2(L+n)(1+\bar{\sigma}^2)\varepsilon]\mathbb{\hat{E}}^{\tilde{G}}_t[\int_{t}^{T}\exp(L(1+\bar{\sigma}^2)(s-t)){\Gamma}^{t,\varepsilon}_{s}ds]\\
\leq &\frac{\exp(L(1+\bar{\sigma}^2)(T-t))}{L}[2\phi(\frac{2L}{n-L})+2(L+n)\varepsilon].
\end{align*}
Sending $\varepsilon\rightarrow 0$,  we have
\[
\hat{Y}_{t} \leq \frac{2\exp(L(1+\bar{\sigma}^2)(T-t))}{L}\phi(\frac{2L}{n-L}),
\]
which completes the proof.
\end{proof}
\begin{remark}{\upshape
Note that from (i) and (ii) of Lemma \ref{lemma2}, in general  we cannot conclude that $\underline{Y}^n$ (or $\bar{Y}^n$) is a Cauchy sequence in $M^{2}_G(0,T)$  according to Proposition \ref{downward convergence proposition},
which is different from the classical case.
}
\end{remark}

Now we are ready to state the main result of this section.

\begin{theorem}\label{main}
Given assumptions \textbf{(H1)-(H4)} and $\xi\in L^{\beta}_G(\Omega_T)$, the $G$-BSDE (\ref{GBSDE1}) admits a unique  solution  $(Y,Z,K)\in \mathfrak{S}_G^{2}(0,T).$
\end{theorem}

\begin{proof}
We shall deal with the existence and uniqueness of solution to $G$-BSDE (\ref{GBSDE1}) separately.  For the  uniqueness, suppose that  both of $(Y^i,Z^i,K^i)$, $i=1,2$ are $\mathfrak{S}_G^{2}(0,T)$-solution to $G$-BSDE (\ref{GBSDE1}), by comparison theorem \ref{my2}, we obtain that for each $n$
\[
\underline{Y}^n_t\leq Y^i_t\leq \bar{Y}^n_t, \ \forall t\in [0,T],
\]
which, together with Lemma \ref{lemma2}, implies
\[
|Y^1_t-Y^2_t|\leq |\bar{Y}^n_t-\underline{Y}^n_t| \leq C_G\phi(\frac{2L}{n-L}), \ \forall n>L.
\]
Note that $Y^i_t$ is a continuous process.
Sending $n\rightarrow\infty$, we deduce that $Y^1=Y^2$ q.s.. Then applying $G$-It\^{o}'s formula upon $\vert Y^1_s-Y^2_s\vert^2$ on $[0,T]$,  we have
$Z^1=Z^2$ and then $K^1=K^2$, which shows that $G$-BSDE (\ref{GBSDE1}) allows for at most one $\mathfrak{S}_G^{2}(0,T)$-solution.

The rest of the proof is devoted  to studying the existence, which will be divided into  three steps. Without loss of generality, we assume $d=1$ and  $g\equiv 0$.

{\it 1 The uniform estimates.}
Let $C(\alpha)$ denote a    constant depending on  parameter $\alpha$  that may change from line to line.
From (ii) of Lemma \ref{lemma2}, we have for all $n$ \[
\|\bar{Y}^n \|_{S^2_G}\leq C(L, \bar{\sigma},\underline{\sigma}, T).\]

Calculating by It\^{o}'s formula upon $\vert \bar{Y}^n \vert^2$, we have for any $t\in[0,T]$,
\begin{align}\label{my3}
\vert \bar{Y}^n_t\vert^2+\int_t^T \vert  \bar{Z}^n_s\vert^2 d\langle B\rangle_s=\vert \xi\vert^2+2\int_t^T\bar{Y}^n_s(\bar{f}_n(s,\bar{Y}^n_s, \bar{Z}^n_s)+f_0(s))ds-2\int_t^T\bar{Y}^n_s \bar{Z}^n_sdB_s-\int_t^T2\bar{Y}^n_sd\bar{K}^n_s.
\end{align}
Since \[
 |\bar{f}^n(t,y,z)|\leq L(1+|y|+|z|),
 \]
we get that
\[
2\bar{Y}^n_s\bar{f}_n(s,\bar{Y}^n_s, \bar{Z}^n_s)\leq 2L(|\bar{Y}^n_s|+|\bar{Y}^n_s|^2)+\frac{4L^2}{\underline{\sigma}^2}|\bar{Y}^n_s|^2+\frac{\underline{\sigma}^2}{4}|\bar{Z}^n_s|^2.
\]
Using BDG inequality and H\"{o}lder's inequality, we derive that
\begin{align*}
 \mathbb{\hat{E}}[\vert \int_0^T \bar{Y}^n_s \bar{Z}^n_sdB_s\vert]\leq C(\bar{\sigma})\mathbb{\hat{E}}[\vert \int_0^T |\bar{Y}^n_s \bar{Z}^n_s|^2ds\vert^{\frac{1}{2}}] &\leq C(\bar{\sigma})\|\bar{Y}^n_s\|_{S^2_G}\|\bar{Z}^n_s\|_{M^2_G}\\
 &\leq C(\bar{\sigma},\underline{\sigma})\|\bar{Y}^n_s\|_{S^2_G}^2+\frac{\underline{\sigma}^2}{8}\|\bar{Z}^n_s\|_{M^2_G}.
\end{align*}
Thus, in view of equation \eqref{my3} we have
\begin{align}\label{my4}
 \mathbb{\hat{E}}[\int_0^T \vert  \bar{Z}^n_s\vert^2 d\langle B\rangle_s]\leq  C(L, \bar{\sigma},\underline{\sigma}, T)+\frac{\underline{\sigma}^2}{2}\|\bar{Z}^n_s\|_{M^2_G}+
 2\mathbb{\hat{E}}[\sup\limits_{s\in[0,T]}|\bar{Y}^n_s||\bar{K}^n_T|].
\end{align}
Recalling that \[
\bar{K}^n_T=\xi-\bar{Y}^n_0+\int_0^T [\bar{f}_n(s,\bar{Y}^n_s, \bar{Z}^n_s)+f_0(s)]ds-\int_0^T  \bar{Z}^n_s dB_s.
\]
By a similar analysis as above,  we obtain
\[
\mathbb{\hat{E}}[\sup\limits_{s\in[0,T]}|\bar{Y}^n_s||\bar{K}^n_T|]\leq  C(L, \bar{\sigma},\underline{\sigma}, T)+\frac{\underline{\sigma}^2}{8}\|\bar{Z}^n_s\|_{M^2_G},
\]
putting together   equation \eqref{my4} with the fact that $\underline{\sigma}^2 \Vert \bar{Z}^n\Vert^2_{M_{G}^2(0,T)} \leq \mathbb{\hat{E}}[\int_0^T \vert  \bar{Z}^n_s\vert^2 d\langle B\rangle_s]$
indicates that
$$ \Vert \bar{Z}^n\Vert_{M_{G}^2(0,T)}+\|\bar{K}^n_T\|_{L^2_G} \leq  C(L, \bar{\sigma},\underline{\sigma}, T), \ \forall n\in\mathbb{N}.$$

{\it 2 The convergence.} From assertions  (i) and (iii) of Lemma \ref{lemma2}, we get that for each $n,m>L$
\[
\|\bar{Y}^n-\bar{Y}^m\|_{S_G^2}\leq \|\bar{Y}^{n\wedge m}-\underline{Y}^{n\wedge m}\|_{S_G^2}\leq  C_G\phi(\frac{2L}{n\wedge m -L})
\]
from which we conclude that $\{ \bar{Y}^n\}_{n\in\mathbb{N}}$  is a Cauchy sequence in $S_G^2(0,T)$.
Then there is a process $Y\in S_G^2(0,T)$ such that $\bar{Y}^n$ converges to $Y$ in $S_G^2(0,T)$.

We continue  to show the convergence of $\bar{Z}^n$ in $ M_G^2(0,T)$.
For each $n,m>L$, applying It\^{o}'s formula to $|\bar{Y}^n-\bar{Y}^m|^2$ yields that
\begin{align*}
&\underline{\sigma}^2\mathbb{\hat{E}}[\int_0^T \vert\bar{Z}^n_s-\bar{Z}^m_s\vert^2 ds]\leq \mathbb{\hat{E}}[\int_0^T \vert\bar{Z}^n_s-\bar{Z}^m_s\vert^2 d\langle B\rangle_s]\\
&\leq 2\mathbb{\hat{E}}[\int_0^T(\bar{Y}^n_s-\bar{Y}^m_s)(\bar{f}_n(s,\bar{Y}^n_s, \bar{Z}^n_s)-\bar{f}_m(s,\bar{Y}^m_s, \bar{Z}^m_s))ds
-\int_0^T(\bar{Y}^n_s-\bar{Y}^m_s)d(\bar{K}^n_s-\bar{K}^m_s)]\\
&\leq  2\mathbb{\hat{E}}[\sup\limits_{s\in[0,T]}|\bar{Y}^n_s-\bar{Y}^m_s|\{L\int_0^T(2+|\bar{Y}^n_s|+|\bar{Y}^m_s|+|\bar{Z}^n_s|+|\bar{Z}^n_s|)ds + |\bar{K}^n_T|+|\bar{K}^m_T|\}]\\
&\leq C(L, \bar{\sigma},\underline{\sigma}, T)\|\bar{Y}^n-\bar{Y}^m\|_{S^2_G},
\end{align*}
where we have used the estimates of {\it step 1} and H\"{o}lder's inequality in the last inequality.
Consequently, we can find some process $Z\in M^2_G(0,T)$ so that  $\bar{Z}^n$ converges to $Z$ in $M_G^2(0,T)$.

Denote
\[
{K}_t:=Y_t-Y_0+\int_0^tf(s,{Y}_s,{Z}_s)ds-\int_0^t{Z}_s dB_s,
\]
we claim that
\begin{align}\label{my5}
\lim_{n\rightarrow\infty}\mathbb{\hat{E}}[\int_{0}^{T}\vert \bar{f}_n(s,\bar{Y}^n_s, \bar{Z}^n_s)+f_0(s)-f(s,{Y}_s,{Z}_s) \vert^{2}ds]=0,
\end{align}
whose proof will be given in {\it step 3}.
Thus it is easy to check that for each $t\in[0,T]$
\[
\lim_{n\rightarrow\infty}\mathbb{\hat{E}}[|{K}_t-\bar{K}^n_t|^2]=0,
\]
which implies that $K$ is a non-increasing $G$-martingale and then $(Y,Z,K)\in\mathfrak{S}_G^{2}(0,T) $ is the solution to $G$-BSDE \eqref{GBSDE1}.

{\it 3 The proof of equation \eqref{my5}.} For each $n>L$, applying lemma \ref{lemma1}, we get that
\begin{align*}
  &\mathbb{\hat{E}}[\int_{0}^{T}\vert \bar{f}_n(s,\bar{Y}^n_s, \bar{Z}^n_s)+f_0(s)-f(s,{Y}_s,{Z}_s) \vert^{2}ds] \\
 & \leq  3\mathbb{\hat{E}}[\int_{0}^{T}\vert \bar{f}_n(s,\bar{Y}^n_s, \bar{Z}^n_s)+f_0(s)-f(s,\bar{Y}^n_s,\bar{Z}^n_s) \vert^{\alpha}ds +\int_{0}^{T}\vert f(s,\bar{Y}^n_s,\bar{Z}^n_s)-f(s,Y_s,\bar{Z}^n_s)\vert^2 ds\\
 &\ \ \ \ +\int_{0}^{T}\vert f(s,{Y}_s,\bar{Z}^n_s)-f(s,Y_s,{Z}_s)\vert^2 ds]
 \\
 & \leq 3 T \phi(\frac{2L}{n-L})+3T L^2 \|\bar{Y}^n-Y\|_{S^2_G}^2+3\mathbb{\hat{E}}[\int_{0}^{T}\vert f(s,{Y}_s,\bar{Z}^n_s)-f(s,Y_s,{Z}_s)\vert^2 ds].
\end{align*}
By the uniform continuity of $f$ in $z$,  for any fixed $\varepsilon >0$, there exists a $\delta$, so that $ \vert f(\cdot,\cdot,x)-f(\cdot,\cdot,y)\vert <\varepsilon$ whenever  $\vert z-q\vert \leq\delta$.
Then for each $N>0$, we obtain that
\begin{align*}
  &\mathbb{\hat{E}}[\int_{0}^{T}\vert f(s,{Y}_s,\bar{Z}^n_s)-f(s,Y_s,{Z}_s)\vert^2 ds]\\ &\leq
  2\mathbb{\hat{E}}[\int_{0}^{T}\vert f(s,{Y}_s,\bar{Z}^n_s)-f(s,Y_s,{Z}_s)\vert^2 \mathbf{1}_{\vert \bar{Z}^n_s-Z_s\vert\leq \delta} ds]
  +2\mathbb{\hat{E}}[\int_{0}^{T}\vert f(s,{Y}_s,\bar{Z}^n_s)-f(s,Y_s,{Z}_s)\vert^2 \mathbf{1}_{\vert \bar{Z}^n_s-Z_s\vert > \delta} ds] \\
   &\leq  2 T \varepsilon^2 + 2\mathbb{\hat{E}}[\int_{0}^{T}(|f_0(s)|+2L\vert Y_s\vert+L\vert \bar{Z}^n_s\vert +L\vert Z_s\vert )^2\mathbf{1}_{\vert \bar{Z}^n_s-Z_s\vert > \delta} ds],
\end{align*}
Since  $\bar{Z}^n$ converges to $Z$ in $M_G^2(0,T)$, it is easy to check that
$\mathbb{\hat{E}}[\int^T_0\mathbf{1}_{\vert \bar{Z}^n_s-Z_s\vert > \delta}]$ is vanishing as $n\rightarrow\infty$.
 Note that $|f_0(s)|+2L\vert Y_s\vert+L\vert \bar{Z}^n_s\vert +L\vert Z_s\vert\in M^2_G(0,T) $. Thus with the help of Theorem 4.7 in \cite{HWZ2016}, we get that
 \[
 \lim_{n\rightarrow\infty}\mathbb{\hat{E}}[\int_{0}^{T}(|f_0(s)|+2L\vert Y_s\vert+L\vert \bar{Z}^n_s\vert +L\vert Z_s\vert )^2\mathbf{1}_{\vert \bar{Z}^n_s-Z_s\vert > \delta} ds]=0.
 \]

 Consequently, putting together the above two inequalities we deduce that
 \[
 \limsup_{n\rightarrow\infty}\mathbb{\hat{E}}[\int_{0}^{T}\vert \bar{f}_n(s,\bar{Y}^n_s, \bar{Z}^n_s)+f_0(s)-f(s,{Y}_s,{Z}_s) \vert^{2}ds] \leq 2 T \varepsilon^2.
 \]
Letting $\varepsilon\rightarrow 0$, we get the desired result.
\end{proof}

\begin{example}{\upshape
For a $1$-dimensional $G$-Brownian motion $B$ with $\underline{\sigma}^2:= -\mathbb{\hat{E}}[-|B_1|^2]$, consider the following $G$-BSDE:
\[
Y_t=\frac{1}{6}|B_T|^6-\frac{5}{2}\underline{\sigma}^2\int^T_t|Z_s|^{\frac{4}{5}}ds-\int^T_t Z_sdB_s-(K_T-K_t).
\]
Note that $f(z)=-\frac{5}{2}\underline{\sigma}^2|z|^{\frac{4}{5}}$ is a uniformly continuous function.
Then by $G$-It\^{o}'s formula and Theorem \ref{main}, it is easy to check that $(\frac{1}{6}|B_t|^6,(B_t)^5,\frac{5}{2}\underline{\sigma}^2\int^t_0 |B_s|^4ds- \frac{5}{2}\int^t_0 |B_s|^4d\langle B\rangle_s)$
is the unique  $\mathfrak{S}_G^{2}(0,T)$-solution.
}
\end{example}

\begin{theorem}[Comparison Theorem]\label{HM7}
Suppose $\xi^\nu\in L^{\beta}_G(\Omega_T)$,  $\nu=1,2$ and $f^\nu,g^{\nu,ij}$ satisfy assumption (H1)-(H4).
Let $(Y^\nu,Z^\nu,K^\nu)$ be the $\mathfrak{S}_G^{2}(0,T)$-solution of $G$-BSDE \eqref{GBSDE1} with data  $(\xi^\nu,f^\nu,g^{\nu,ij})$.
If $\xi^1 \leq \xi^2$, $f^1(t,y,z)\leq f^2(t,y,z)$ and $g^{1,ij}(t,y,z)\leq g^{2,ij}(t,y,z)$ for any $(t,\omega,y,z)$,  then
$Y^1_t\leq Y^2_t$ for each $t$.
\end{theorem}
\begin{proof}
For each $n\in\mathbb{N}$, let $(\bar{Y}^{2,n},\bar{Z}^{2,n},\bar{K}^{2,n} )$ be the  $\mathfrak{S}_G^{2}(0,T)$-solution of $G$-BSDE \eqref{my1} corresponding to $(f^2,g^{2,ij})$.
It is obvious that $ \bar{f}^2_n(t,y,z)+f^2_0(t)\geq f^1(t,y,z)$ and $\bar{g}_n^{2,ii}(t,y,z)+g^{2,ii}_0(t)\geq g^{1,ii}(t,y,z)$. Note that $g^{2,ij}=g^{1,ij}=\bar{g}_n^{2,ij}=0$ whenever $i\neq j$.
Then using Theorem \ref{my2}, we get that ${Y}^1_t\leq \bar{Y}^{2,n}_t$ for each $t$.
Note that $\bar{Y}^{2,n}$ converges to $Y^2$ in $S^2_G(0,T)$.
Sending $n\rightarrow\infty$, we derive that $Y^1_t\leq Y^2_t$. The proof is complete.
\end{proof}

\section{Nonlinear  Feynman-Kac formula}
In this section, we shall utilize  Theorem  \ref{main} to establish a nonlinear  Feynman-Kac formula that slightly generalizes the corresponding result of \cite{HJPS2014C,P10}.
Retaining the notations in previous sections,  for each $(t,x)\in[0,T]\times\mathbb{R}^m$, let's consider    $G$-BSDE
\begin{align}
  dY_s^{t,x}=&-f(s,X_s^{t,x},Y_s^{t,x},Z_s^{t,x})ds-g^{i}(s,X_s^{t,x},Y_s^{t,x},Z_s^{t,x})d\langle B^i,B^i\rangle_s+Z_s^{t,x}dB_s+dK_s^{t,x}, \ s\in[t,T] \notag\\ Y_T^{t,x}=&\Phi(X_T^{t,x}),\label{fbsde2}
\end{align}
where $X^{t,x}$ is defined through  a forward $G$-SDE on the interval $[t,T]$
\begin{equation}\label{fbsde1}
  dX_s^{t,x}=b(s,X_s^{t,x})ds+h^{ij}(s,X_s^{t,x})d\langle B^i,B^j\rangle_s+\sigma(s,X_s^{t,x})dB_s, \ \ X_t^{t,x}=x.
\end{equation}

In the sequel, we use these   running assumptions abbreviated as \textbf{(H5)} :
\begin{itemize}
  \item[(i)] $b,h^{ij}=h^{ji}:[0,T]\times\mathbb{R}^m\rightarrow \mathbb{R}^m; \sigma:[0,T]\times\mathbb{R}^{m}\rightarrow \mathbb{R}^{m\times d}; f,g^{i}:[0,T]\times\mathbb{R}^m\times\mathbb{R}\times\mathbb{R}^d\rightarrow\mathbb{R};\Phi:\mathbb{R}^m\rightarrow \mathbb{R}, $ are all deterministic continuous functions.
  \item[(ii)]There exist two positive integers $q, L$ and a  modulus of continuity $\phi$ such that
  \begin{align*}
     &\vert b(t,x)-b(t,x')\vert+\sum_{i,j=1}^{d}\vert h^{ij}(t,x)-h^{ij}(t,x')\vert
    +\vert \sigma(t,x)-\sigma(t,x')\vert \leq L\vert x-x'\vert,\\
    &\vert\Phi(x)-\Phi(x')\vert +\vert f(t,x,y,z)-f(t,x',y',z')\vert+\sum_{i=1}^{d}\vert g^{i}(t,x,y,z)-g^{i}(t,x',y',z')\vert\\
   &\leq L(1+\vert x\vert^q+\vert x'\vert^q)\vert x-x'\vert+L\vert y-y'\vert+ \phi(\vert z-z'\vert).
  \end{align*}
\end{itemize}

To link the above $G$-BSDE system with  PDE, we need several estimates
 from \cite{HJPS2014C,P10},
\begin{lemma}\label{estimateX}
  Assuming  \textbf{(H5)}, for any $\delta \in[0,T-t]$, there exists a constant $C$ depending on $L',G,p,n,T$ such that
  \begin{align*}
    & \mathbb{\hat{E}}_t[\vert X_{t+\delta}^{t,x}\vert^p]\leq C(1+\vert x\vert^p),\\
&\mathbb{\hat{E}}_t[\vert X_{t+\delta}^{t,x}-X_{t+\delta}^{t,x'} \vert^p]\leq C\vert x-x'\vert^p,\\
&\mathbb{\hat{E}}_t\Big[\sup_{s\in[t, t+\delta]}\vert X_s^{t,x}-x\vert^p\Big]\leq C(1+\vert x\vert^p)\delta^{p/2}.
  \end{align*}
\end{lemma}

\begin{theorem}
Suppose  \textbf{(H5)} hold. Then $G$-BSDE \eqref{fbsde2} has a unique solution triplet  $(Y^{t,x},Z^{t,x},K^{t,x})\in \mathfrak{S}_G^{2}(t,T).$
\end{theorem}
\begin{proof}
By Lemma \ref{estimateX} and assumption \textbf{(H5)}, for each $p\geq 1,$
it is easy to get that $\Phi(X_T^{t,x})\in L^p_G(\Omega_T)$. The facts $f(.,X_.^{t,x},y,z)$, $g^{ij}(.,X_.^{t,x},y,z)\in M^p_G(t,T)$  follow from Theorem 4.16 in \cite{HWZ2016}. Therefore it suffices to verify conditions \textbf{(H4)}, before applying Theorem \ref{main} to complete the proof.

For  any $(t,x,y,z)\in [0,T]\times\mathbb{R}^m\times\mathbb{R}\times\mathbb{R}^d$ and $n\in\mathbb{N}$,
set
\begin{align*}
\underline{\varphi}_n(t,x,y,z)& =\inf\limits_{q\in\mathbb{Q}}\{\varphi(t,x,y,q)+n\vert z-q\vert\}-\varphi_0(t,x),\\
\bar{\varphi}_n(t,x,y,z)& =\sup\limits_{q\in\mathbb{Q}}\{\varphi(t,x,y,q)-n\vert z-q\vert\}-\varphi_0(t,x),
\end{align*}
for $\varphi=f, g^{i}$ and $\varphi_0(t,x)=\varphi(t,x,0,0)$.
By property (ii) of \textbf{(H5)}, we derive that
\[
|\underline{\varphi}_n(t,x,y,z)-\underline{\varphi}_n(t,x^{\prime},y,z)|+|\bar{\varphi}_n(t,x,y,z)-\bar{\varphi}_n(t,x^{\prime},y,z)|\leq 4L(1+\vert x\vert^q+\vert x'\vert^q)\vert x-x'\vert,
\]
which, together with Theorem 4.16 in \cite{HWZ2016} and Lemma \ref{estimateX},  implies that both
$\underline{\varphi}_n(.,X_.^{t,x},y,z)$ and $\bar{\varphi}_n(.,X_.^{t,x},y,z)$ belong to  $M^p_G(t,T)$ for each $p\geq 1.$
\end{proof}

Using the same notations appearing  in the above argument,
for each $(t,x)\in [0,T]\times\mathbb{R}^m$ and $n\in \mathbb{N}$, we consider a sequence of approximating $G$-BSDEs corresponding respectively to generators $(\underline{f}_n,\underline{g}_{n}^{ij})$ and $(\bar{f}_n,\bar{g}_{n}^{ij})$ on $[t,T]$,
\begin{align}
\begin{split}
\underline{Y}^{n,t,x}_s=&\Phi(X^{t,x}_T) +
\int_s^T [\underline{f}_n(r,X^{t,x}_r,\underline{Y}^{n,t,x}_r, \underline{Z}^{n,t,x}_r)+f_0(r,X^{t,x}_r)]dr-\int_s^T  \underline{Z}^{n,t,x}_r dB_r
 \\
& \ \ \ \ \ +\int_s^T[ \underline{g}^{i}_n(r,X^{t,x}_r,\underline{Y}^{n,t,x}_r, \underline{Z}^{n,t,x}_r)+g^{i}_0(r,X^{t,x}_r)]d\langle B^i,B^i\rangle_r - (\underline{K}^{n,t,x}_T-\underline{K}^{n,t,x}_s),\\
 \bar{Y}^{n,t,x}_s=&\Phi(X^{t,x}_T) +\int_s^T [\bar{f}_n(r,X^{t,x}_r,\bar{Y}^{n,t,x}_r, \bar{Z}^{n,t,x}_r)+f_0(r,X^{t,x}_r)]dr -\int_t^T  \bar{Z}^{n,t,x}_s dB_s
 \\
& \ \ \ \ \ +\int_s^T [\bar{g}_n^{i}(r,X^{t,x}_r,\bar{Y}^{n,t,x}_r, \bar{Z}^{n,t,x}_r)+g^{i}_0(r,X^{t,x}_r)]d\langle B^i,B^i\rangle_r -(\bar{K}^{n,t,x}_T-\bar{K}^{n,t,x}_s).
\end{split}
\end{align}
If we denote
$$ \underline{u}^n(t,x):=\underline{Y}^{n,t,x}_t, \ \ \bar{u}^n(t,x):=\bar{Y}^{n,t,x}_t, \ \ (t,x)\in [0,T]\times\mathbb{R}^m.$$
By Proposition 4.2 in \cite{HJPS2014C}, both $\underline{u}^n$ and $\bar{u}^n$ are continuous functions. Similarly we can define
$$ u(t,x):=Y^{t,x}_t, \ \ (t,x)\in [0,T]\times\mathbb{R}^m.$$
Clearly $u(t,x)$ is a well-defined deterministic function from the above theorem. And some regularity can be derived from that of $\underline{u}^n,\bar{u}^n$,
indeed we have
\begin{lemma} \label{my8}
Given assumption \textbf{(H5)}, $u$ is a continuous function of polynomial growth.
\end{lemma}

\begin{proof} Without loss of generality, assume that $|\phi(z)|\leq L(1+|z|)$.
Setting $w(y,z)=L(1+|y|+|z|)$,  consider the following $G$-BSDEs on $[t,T]$
\begin{align*}
\begin{split}
&U_s=\Phi(X^{t,x}_T) +\int_s^T [w( U_r, V_r)+f_0(r,X^{t,x}_r)] dr+\int_s^T [w( U_r, V_r)+g_0^{i}(r,X^{t,x}_r)] d\langle B^i,B^i\rangle_r\\
&\ \ \ \ \ \ \ \ \ \ \ \ \ \ -\int_s^T  V_r dB_r- (R_T-R_s),\\
&U'_s=\Phi(X^{t,x}_T) +\int_s^T [-w( U'_r, V'_r)+f_0(r,X^{t,x}_r)] dr+\int_s^T [-w( U'_r, V'_r)+g_0^{i}(r,X^{t,x}_r)] d\langle B^i,B^i\rangle_r\\
&\ \ \ \ \ \ \ \ \ \ \ \ \ \ -\int_s^T  V'_r dB_r- (R'_T-R'_s).
\end{split}
\end{align*}
By (i) of lemma \ref{lemma1} and the comparison theorem \ref{my2}, we have for each fixed $(t,x)$
\[
U'_s\leq \underline{Y}^{n,t,x}_s\leq Y^{t,x}_s\leq\bar{Y}^{n,t,x}_s \leq U_s, \ \forall s\in[t,T], \]
Recalling Proposition 4.2 in \cite{HJPS2014C}, we can find some constant $\bar{C}$ depending on $L, G,m$ and $T$ so that
\[
|U'_t|+|U_t|\leq \bar{C}(1+|x|^{q+1}),
\]
which indicates that $u$ is of   polynomial growth.

Applying (iv) of lemma \ref{lemma1} yields that for each $(t,x)\in[0,T]\times\mathbb{R}^m$,
$$\vert \underline{u}^n(t,x)-{u}(t,x)\vert\leq C_G\phi(\frac{2L}{n-L}), \ \ n>L,$$
i.e., $\underline{u}^n(t,x)$ converges to ${u}(t,x)$ uniformly in $(t,x)$. Consequently, $u$ is continuous in $(t,x)$, which ends the proof.
\end{proof}

The main result of this section is,
\begin{theorem} \label{my6}
 Let assumption \textbf{(H5)} be given. Then  $u$ is the
unique viscosity solution to the following PDE:%
\begin{align}
\begin{cases}
&\partial_tu+G(H(t,x,u,D_xu,D_{x}^{2}u))+\langle
b(t,x),D_{x}u\rangle+f(t,x,u,\sigma^{\mathrm{\top}}(t,x)D_{x}u)=0,\\
&u(T,x)=\Phi(x),\ \  \ \ x\in \mathbb{R}^m,
\end{cases}
\label{feynman}
\end{align}
where
\[
\begin{array}
[c]{cl}
H_{ij}(t,x,v,p,A)=(\sigma^{\mathrm{\top}}(t,x)A\sigma(t,x))_{ij}+2\langle
p,h^{ij}(t,x)\rangle
 +2g^{i}(t,x,v,p\sigma(t,x))\mathbf{1}_{\{i=j\}}
\end{array}
\]
for any $(t,x,v,p,A)\in[0,T]\times\mathbb{R}^{m}\times\mathbb{R}%
\times\mathbb{R}^{m}\times\mathbb{S}_{m}$.
\end{theorem}

For reader's convenience, we provide with the definition of  viscosity solution to equation \eqref{feynman}, see \cite{CMI}.
For every  $v\in C([0,T]\times\mathbb{R}^m)$, denote by $\mathcal{P}^{2,+}v(t,x)$ the
\textquotedblleft {parabolic superjet}\textquotedblright \ of $v$ at $(t,x)$, which refers to  the set of triples
$(a,p,X)\in \mathbb{R}\times \mathbb{R}^{m}\times \mathbb{S}_m$
such
that%
\begin{align*}
v(s,y)  &  \leq v(t,x)+a(s-t)+\langle p,y-x\rangle +\frac{1}{2}\langle X(y-x),y-x\rangle+o(|s-t|+|y-x|^{2}).
\end{align*}
Similarly the
\textquotedblleft
{parabolic subjet}\textquotedblright \ of $v$ at $(t,x)$ can be defined  by $\mathcal{P}
^{2,-}v(t,x):=-\mathcal{P}^{2,+}(-v)(t,x)$.

\begin{definition}
\textup{(i)}  For $v\in C([0,T]\times \mathbb{R}^{m})$, $v$ is  a viscosity subsolution of \eqref{my6} on $[0,T]\times \mathbb{R}^{m}$, if $v(T,x)\leq \Phi(x)$ and for all $(t,x)\in(0,T)\times \mathbb{R}^{m}$,
\[
a+G({H}(t,x,{v}(t,x),p,X))+\langle b(t,x),p\rangle+{f}(t,x,{v}(t,x),p\sigma(t,x))\geq 0, \  \text{for }\ (a,p,X)\in \mathcal{P}
^{2,+}v(t,x).
\]
\textup{(ii)} A {viscosity supersolution} of equation \eqref{my6} on $[0,T]\times \mathbb{R}^{m}$ refers to function
$v\in C([0,T]\times \mathbb{R}^{m})$ with  $v(T,x)\geq \Phi(x)$ such that for each
$(t,x)\in(0,T)\times \mathbb{R}^{m}$,
\[
a+G({H}(t,x,{v}(t,x),p,X))+\langle b(t,x),p\rangle+{f}(t,x,{v}(t,x),p\sigma(t,x))\leq 0, \  \text{for }\ (a,p,X)\in \mathcal{P}
^{2,-}v(t,x).
\]

 A function $v\in C([0,T]\times \mathbb{R}^{m})$ is called a {viscosity solution} of equation \eqref{my6}
if it is simultaneously a viscosity subsolution and a viscosity supersolution of equation \eqref{my6} on $[0,T]\times \mathbb{R}^{m}$.
\end{definition}

\begin{proof}[The proof of Theorem \ref{my6}]
Since   the uniqueness of  viscosity solution to equation  \eqref{my6} is well established, c.f. \cite{PE, P10}, by the symmetry of supsolution and subsolution, we only check that $u$ is a viscosity subsolution.

Given $(t,x)\in(0, T)\times\mathbb{R}^n$ and $(a,p,X)\in\mathcal{P}^{2,+}u(t,x)$, since $\underline{u}^n$ converges to $u$ uniformly  in $(t,x)$, we get that
\begin{align*}
\lim\limits_{n\rightarrow\infty}|\underline{u}^n(t_n,x_n)-u(t,x)|=0,
\end{align*}
whenever $(t_n,x_n)\rightarrow (t,x)$.
With the help of Proposition 4.3 in \cite{CMI}, there exist  sequences
\begin{align*}
n_k\rightarrow\infty,  \ (t_k,x_k)\rightarrow(t,x),\  \text{and}\ (a_k,p_k,X_k)\rightarrow(a,p,X) ,
\end{align*}
such that
\begin{align*}
  (a_k,p_k,X_k)\in\mathcal{P}^{2,+}\underline{u}^{n_k}(t_k,x_k).
\end{align*}

From the Feynman-Kac formula in \cite{HJPS2014C},  we know $\underline{u}^n(t,x)$ is the unique viscosity solution to
\begin{align*}
\begin{cases}
&\partial_t\underline{V}+G(\underline{H}^n(t,x,\underline{V},D_x\underline{V},D_{x}^{2}\underline{V}))+\langle
b(t,x),D_{x}\underline{V}\rangle+\underline{f}_n(t,x,\underline{V},\sigma^{\top}(t,x)D_{x}\underline{V})+f_0(t,x)=0,\\
&\underline{V}(T,x)=\Phi(x),\ \  \ \ x\in \mathbb{R}^m,
\end{cases}
\end{align*}
with
\[
\begin{array}
[c]{cl}
\underline{H}^n_{ij}(t,x,v,p,A)=(\sigma^{\top}(t,x)A\sigma(t,x))_{ij}+2\langle
p,h^{ij}(t,x)\rangle
 +2\underline{g}_n^{i}(t,x,v,p\sigma(t,x))\mathbf{1}_{\{i=j\}}+2g^{i}_0(t,x)\mathbf{1}_{\{i=j\}}.
\end{array}
\]
Thus by the definition of viscosity solution we derive that
\[
a_k+G(\underline{H}^n(t_k,x_k,\underline{u}^{n_k}(t_k,x_k),p_k,X_k))+\langle b(t_k,x_k),p_k\rangle+\underline{f}_n(t_k,x_k,\underline{u}^{n_k}(t_k,x_k),p_k\sigma(t_k,x_k))+f_0(t_k,x_k)\geq 0.
\]
Recalling (i) of lemma \ref{lemma1}, we obtain that for $\varphi=f, g^{i}$
\[
\varphi(t,x,y,z)\geq \underline{\varphi}_n(t,x,y,z)+\varphi_0(t,x), \ \text{ for  all }   (t,x,y,z)\in[0,T]\times\mathbb{R}^m\times\mathbb{R}\times\mathbb{R}^{d},
\]
which implies that $H(t,x,v,p,A)\geq \underline{H}^n(t,x,v,p,A)$.
Then we derive that
\[
a_k+G({H}(t_k,x_k,\underline{u}^{n_k}(t_k,x_k),p_k,X_k))+\langle b(t_k,x_k),p_k\rangle+{f}(t_k,x_k,\underline{u}^{n_k}(t_k,x_k),p_k\sigma(t_k,x_k))\geq 0.
\]
Sending $k\rightarrow\infty$, we get
\[
a+G({H}(t,x,{u}(t,x),p,X))+\langle b(t,x),p\rangle+{f}(t,x,{u}(t,x),p\sigma(t,x))\geq 0,
\]
which is the desired result.
\end{proof}

\end{document}